\documentclass[11pt]{article}
\usepackage{amsmath,amsfonts,amssymb}
\usepackage{amsthm}
\usepackage{mathrsfs}
\usepackage{enumerate}
\usepackage{hyperref}
\usepackage{pdfsync}
\usepackage{pdfpages}
\usepackage{graphicx}
\usepackage{float}
\usepackage{cleveref}

\newcommand{\R}{\ensuremath{\mathbb{R}}}

\newcommand{\Z}{\ensuremath{\mathbb{Z}}}

\theoremstyle{plain}           	
\newtheorem{theorem}{Theorem}
\newtheorem{lemma}[theorem]{Lemma}

\newtheorem{conjecture}[theorem]{Conjecture}

\theoremstyle{remark}

\DeclareMathOperator*{\E}{E}

\DeclareMathOperator*{\p}{P}

\title{\vspace{-0.9cm}A remark on Hamilton cycles with few colors}
\date{}
\author{Igor Balla\thanks{Department of Mathematics, ETH, 8092 Zurich.} \textsuperscript{,}\thanks{igor.balla@math.ethz.ch.}
\and Alexey Pokrovskiy\footnotemark[1] \textsuperscript{,}\thanks{dr.alexey.pokrovskiy@gmail.com.} 
\and Benny Sudakov\footnotemark[1] \textsuperscript{,}\thanks{benjamin.sudakov@math.ethz.ch.} }

\begin{document}

\maketitle 

\begin{abstract}
Akbari,  Etesami, Mahini, and Mahmoody conjectured that every proper edge colouring of $K_n$  with $n$ colours contains a Hamilton cycle with $\leq O(\log n)$ colours. They proved that there is always a Hamilton cycle with $\leq 8\sqrt n$ colours. In this note we improve this bound to $O(\log^3 n)$.
\end{abstract}

An edge-colouring of $K_n$ is proper if no two edges of the same colour share a vertex. 
The smallest number of colours needed to properly edge-colour a graph $G$ is called the chromatic index of $G$, denoted by $\chi'(G)$. It is well known that $\chi'(K_n)=n-1$ if $n$ is even and $\chi'(K_n)=n$ is $n$ is odd.

Akbari,  Etesami, Mahini, and Mahmoody~\cite{AEMM} investigated cycles in properly edge coloured complete graphs. Specifically they looked for Hamilton cycles in properly coloured complete graphs which have either many or few colours (a Hamilton cycle in a graph is one which passes through every vertex.) When looking for Hamilton cycles with \emph{few} colours it is natural to bound the total number of colours in the properly coloured $K_n$. Otherwise, if one looks at a properly coloured $K_n$ with $\binom n2$ colours, then every Hamilton cycle trivially has exactly $n$ colours. Because of this Akbari et al. looked at properly coloured $K_n$ with $\chi'(K_n)$ colours. They made the following conjecture about how few colours one can have on a Hamilton cycle in such a colouring.

\begin{conjecture}[Akbari,  Etesami, Mahini, and Mahmoody~\cite{AEMM}]\label{Conjecture_AEMM}
Every properly edge-coloured $K_n$ with $\chi'(K_n)$ colours has a Hamilton cycle with $\leq O(\log n)$ colours. 
\end{conjecture}

To see that there are proper $\chi'(K_n)$-edge-coloured $K_n$  with no Hamilton cycles with less than $\log n$ colours, consider a colouring of the edges of the complete graph with vertex set $\Z_2^k$, where the edge $ij$ is coloured by colour $i+j$. Indeed, any Hamilton cycle of this graph contains $0$ and any other vertex $i$ is a sum of the colours of the edges on the path from $0$ to $i$. Thus the number of edge colours must be at least $\dim(\Z_2^k) = k$.

Towards Conjecture~\ref{Conjecture_AEMM}, Akbari et al. proved that every properly $\chi'(K_n)$-edge-coloured $K_n$ has a Hamilton cycle with $8\sqrt n$ colours~\cite{AEMM}. 
In this note we explain how to improve this to $O(\log^3 n)$.
\begin{theorem} \label{main}
For any sufficiently large $n$, any properly edge-colored $K_n$ with $\chi'(K_n)$ colors contains a Hamilton cycle with at most $O(\log^3{n})$ colours. 
\end{theorem}

To prove Theorem~\ref{main} we select a set of $\log^3 n$ colours at random and show that, with high probability, the subgraph consisting of these colours is Hamiltonian. The Hamiltonicity of this subgraph follows from the proof of Theorem 10 in Christofides and Markstrom \cite{CMLatinSquare}. First we show that, with high probability, all eigenvalues of this graph except the first one are small in absolute value, so that the graph is pseudo-random. Then a result of Krivelevich and Sudakov \cite{KS03} implies that such graphs are Hamiltonian. 

Given a $d$-regular graph $H$ with vertex set $\{1, \ldots, n\}$, let $A$ be the corresponding adjacency matrix, i.e.\ an $n \times n$ matrix such that $A_{i,j} = 1$ if $ij \in E(H)$ and $A_{i,j} = 0$ otherwise. Let $\lambda_1 \geq \lambda_2 \geq \ldots \geq \lambda_n$ be the eigenvalues of $A$. Then we have $d = \lambda_1$ and $|\lambda_i | \leq d$ for all $i \in [n]$. We define $\lambda(H) = \max_{i=2}^n{|\lambda_i |}$ and note that the smaller $\lambda(H)$ is, the more pseudo-random our graph $H$ is. For any $\lambda > 0$, we say that $H$ is an $(n,d,\lambda)$-graph if $\lambda(H) \leq \lambda$.

\begin{theorem}[Krivelevich and Sudakov \cite{KS03}] \label{Hamilton}
Let $H$ be an $(n,d,\lambda)$-graph. If $n$ is large enough and 
\[ \lambda \leq \frac{(\log{\log{n}})^2}{1000 \log{n} ( \log{\log{\log{n}}} )} d,\]
then $H$ is Hamiltonian.
\end{theorem}

Thus it will suffice to show that by choosing $d = \log^3 n$ colours at random, the graph $H$ we obtain is an $(n, d, \frac{2 \log{\log{n}}}{\log{n}}d)$-graph  with high probability. Actually, $H$ will only be regular when $n$ is even and so we first prove the result for even $n$.

\begin{lemma} \label{mainlemma}
For any sufficiently large even $n$, any properly edge-colored $K_n$ with $n-1$ colors contains a Hamilton cycle with at most $\log^3{n}$ colours. 
\end{lemma}

The case of odd $n$ can then be obtained from the above lemma as follows.

\begin{proof}[Proof of \Cref{main}]
The case of even $n$ is given by \Cref{mainlemma}. So let $n$ be odd and let a proper $n$ edge-colouring of $K_n$ be given. Note that each vertex $u$ has a color $c(u)$ which is not used by any of the edges incident to it. Moreover, since there are $n$ colors, each color $a$ must appear $(n-1)/2$ times and hence there is some vertex $u$ such that $c(u) = a$. Thus if we add a new vertex $v$ and for each $u \in K_n$ color the edge $uv$ with $c(u)$, we obtain a proper $n$ edge-colouring of $K_{n+1}$. By \Cref{mainlemma}, it has a Hamilton cycle with at most $\log^3(n+1)$ colours, which forms a Hamilton path on $K_n$. By connecting the endpoints of the path (possibly using an extra colour), we obtain a Hamilton cycle with at most $\log^3(n+1) + 1$ colours.
\end{proof}

The fact that $\lambda(H)$ is small will follow from an operator Hoeffding inequality for Hilbert spaces, obtained by Christofides and Markstrom~\cite{CMRandomCayley} (extending the work of \cite{AW}). To this end, let $V$ be a Hilbert space of dimension $d$ and $S(V)$ be the set of self adjoint operators on $V$. For any $A,B \in S(V)$, we define $A \leq B$ iff $B-A$ is positive semidefinite and define $[A, B] = \{C \in S(V) : A \leq C \leq B\}$. We let  $||A|| = \sup_{|v| = 1}{|Av|}$ denote the operator norm, where $|v| = \sqrt{\langle v, v \rangle}$.

\begin{theorem}[Operator Hoeffding \cite{CMRandomCayley, CMLatinSquare}] \label{OpHoeffding}
Let $V$ be a Hilbert space of dimension $n$ and let $0 = X_0, X_1, \ldots, X_d$ be a martingale taking values in $S(V)$, such that $X_i - X_{i-1} \in \left[-\frac{1}{2}I, \frac{1}{2}I\right]$ for all $i \in [d]$. Then for $0 < t < 1/2$,
\[ \p\left[||X_d - \E[X_d] || \geq dt\right] \leq 2n \exp(- 2 d t^2 ).\]
\end{theorem}

\begin{proof}[Proof of \Cref{mainlemma}]
Given a proper edge-colouring $c:E(K_n)\rightarrow [n-1]$, let $d = \log^3{n}$ and let $c_1, \ldots, c_d$ be a sequence of colours chosen independently and uniformly at random from $[n-1]$. For each $l \in [d]$, let $A^{(l)}$ be the adjacency matrix of color $c_l$ with probability $\frac{n-1}{n}$ and $A^{(l)}=I$ with probability $1/n$. Let $A  = A^{(1)} + \ldots + A^{(d)}$ and note that the probability that all $d$ colours are distinct and all $A^{(l)}\not =I$ is $\frac{n-1}{n} \frac{n-2}{n} \ldots \frac{n-d}{n}=1-o(1)$. In this case, $A$ is the adjacency matrix of a simple, $d$-regular graph $H$.

Furthermore, we claim that $\lambda(H) = ||A - \frac{d}{n}J||$ where $J$ is the all 1's matrix. Indeed, letting $v_1, \ldots, v_n$ be an orthonormal basis of eigenvectors of $A$ ($v_1$ is the all 1's vector), with corresponding eigenvalues $\lambda_1 \geq \ldots \geq \lambda_n$, we have $(A - \frac{d}{n}J)v_1 = 0$ and $(A - \frac{d}{n}J)v_i = \lambda_i v_i$ for all $i \geq 2$.

Also note that for each $l \in [d]$, the eigenvalues $\mu_1 \geq \ldots \geq \mu_n$ of $A^{(l)}$ satisfy $\mu_1 = 1$ and $|\mu_i| \leq 1$ for $i \geq 2$. Moreover, $v = (1, \ldots, 1)$ is an eigenvector of $\mu_1$ and hence $Y_l = \frac{1}{2}(A^{(l)} - \frac{1}{n}J)$ has eigenvalues $0, \mu_2 / 2, \ldots, \mu_n / 2$, so that $Y_l\in \left[-\frac{1}{2}I,\frac{1}{2}I\right]$. Moreover, we have $E[Y_l] = 0$ which implies that $X_i = Y_1 + \ldots + Y_i$ is a martingale. Note also that $X_d = \frac{1}{2}(A - \frac{d}{n}J)$ and $X_i\in S(\R^n)$ for all $i$.  Thus, we can apply \Cref{OpHoeffding} with $t = \log{\log{n}} / \log{n}$ to conclude that
\[ \p\left[ \frac{1}{2}\lambda(H) \geq dt\right] =  \p\left[ ||X_d ||  \geq dt\right] \leq 2n \exp(-2d t^2) = 2n^{1 - 2 (\log{\log{n}})^2 } \rightarrow 0\]
as $n \rightarrow \infty$. Thus for $n$ large enough, we have that with high probability
\[ \lambda(H) \leq 2dt \leq \frac{(\log{\log{n}})^2}{1000 \log{n} ( \log{\log{\log{n}}} )} d, \]
and so in this case, we may apply \Cref{Hamilton} to conclude that $H$ has a Hamilton cycle.
\end{proof}

\noindent
{\bf Remark.}\, One can deduce the Hamiltonicity of the random set of $\log^3 n$ colours directly from the statement of Theorem
10 in \cite{CMLatinSquare}. We choose not to do so for the convenience of the reader and since the proof of Christofides and Markstrom in \cite{CMLatinSquare}
needs a version of Theorem \ref{Hamilton} that works for multi-graphs with self-loops. Although such a theorem should have a very similar proof to that in \cite{KS03}, it does not appear in the literature. We avoid this issue by observing that the graph we obtain is simple with high probability.

\end{document}